\documentclass[11pt,a4paper,final]{amsart}

\usepackage{amssymb,latexsym}
\usepackage{enumerate}
\usepackage{pifont,wasysym}
\usepackage[T1]{fontenc}

\usepackage{shuffle}
\usepackage[left=3cm,right=3cm,top=3cm,bottom=3cm]{geometry} 
\usepackage{hyperref} 

\theoremstyle{plain}
\newtheorem{theorem}{Theorem}
\newtheorem{corollary}{Corollary}[section]
\newtheorem{lemma}[corollary]{Lemma}
\newtheorem{proposition}[corollary]{Proposition}

\theoremstyle{remark}
\newtheorem{remark}[corollary]{Remark}

\numberwithin{equation}{section}

\newcommand{\NN}{\mathbb{N}}

\newcommand{\RR}{\mathbb{R}}

\newcommand{\Set}[1]{\left\{ #1  \right\}}
\newcommand{\SetSuchThat}[2]{\left\{\, #1 \ | \ #2 \, \right\}}
\newcommand{\Cardinality}[1]{\# #1}

\newcommand{\ScalarProduct}[3]{(#1| #2)_{#3}}

\newcommand{\TangentFields}[1]{\Gamma(#1)}
\newcommand{\LieAlgebra}[1]{\mathfrak{#1}}
\newcommand{\VectorFlow}[1]{\exp\left(#1\right)}

\newcommand{\Cases}[1]{
\left\{ \begin{array}{ll}
#1 
\end{array}\right.
}

\newcommand{\Case}[2]{ #1 & \text{for } #2}

\newcommand{\Manifold}{{\mathfrak{M}}}

\newcommand{\Solution}{\hat x}
\newcommand{\MainSeries}{\hat{S}}
\newcommand{\IteratedInt}[1]{\Upsilon^{#1}}
\newcommand{\Domain}[1]{D(#1)}

\newcommand{\MultiIndex}[1]{\mathbf #1}
\newcommand{\MultiIndexSet}[1]{M(#1)}
\newcommand{\MultiIndexSetOrdered}[1]{M_{\leq}(#1)}
\newcommand{\MultiIndexSetZero}[1]{M_0(#1)}
\newcommand{\NumberOfWords}{\#}
\newcommand{\Number}{N}

\newcommand{\NS}[2]{{#1}_{#2}}
\newcommand{\NewtonSymbol}[2]{
\left(\!\!  \begin{array}{c}
#1 \\
#2
\end{array} \!\!\right) 
}

\newcommand{\NewtonSymbolGeneral}[3]{\frac{{#1}!}{{#2}_1 !\cdots {#2}_{#3} !}}

\newcommand{\Trees}[2]{\mathcal{T}^{#1}_{#2}}
\newcommand{\Tree}{\mathsf {t}}
\newcommand{\Vertex}{\mathsf {v}}

\newcommand{\ShuffleProduct}{\shuffle}

\newcommand{\EmptyWord}{\mathrm{1}}
\newcommand{\ExpShuffle}[1]{\exp_{\ShuffleProduct}(#1)}

\newcommand{\Polynomials}[2]{#1\langle#2\rangle}
\newcommand{\Series}[2]{#1\langle\langle#2\rangle\rangle}
\newcommand{\WordsInSeriesOfLength}[2]{S_{#2}^{#1}}

\newcommand{\AOneSeries}{S_{\LieAlgebra{a}_1}}

\newcommand{\Letters}{\mathrm{A}}
\newcommand{\WordsOfLength}[1]{\Letters^*_{#1}}

\begin{document}

\title[Expansion of Solution of Non-autonomous Polynomial Differential Equation]{On Expansion of a Solution of General Non-autonomous Polynomial Differential Equation}

\author[G. Pietrzkowski]{Gabriel Pietrzkowski$^\dagger$ }
\thanks{$\dagger$ Institute of Applied Mathematics and Mechanics, Faculty of Mathematics, Informatics and Mechanics, University of Warsaw; Banacha 2, 02-097 Warszawa, Poland; \phone +48-22-55 44 524.}

\address{Institute of Applied Mathematics and Mechanics, Faculty of Mathematics, Informatics and Mechanics, University of Warsaw; Banacha 2, 02-097 Warszawa, Poland; \phone +48-22-55 44 524.}

\email{G.Pietrzkowski@mimuw.edu.pl}

\begin{abstract}
We give a recursive formula for an expansion of a solution of a general non-autonomous polynomial differential equation. The formula is given on the algebraic level with a use of shuffle product. This approach minimizes the number of integrations on each order of expansion. Using combinatorics of trees we estimate the radius of convergence of the expansion.

\end{abstract}



\maketitle

\section{Introduction}
\label{sec:Introduction}

Consider a non-autonomous polynomial differential equation, known also as a generalized Abel differential equation 
\begin{align}
\label{eq:PreMain}
\dot x(t) &= u_0(t) +  u_1(t)x(t) +\cdots+ u_i(t)x^i(t) + \cdots +  u_n(t) x^n(t), \\
\nonumber
x(0) &= x_0,
\end{align}
with a solution $x:[0,T]\to \RR$ on a small segment of the reals. In this class of differential equation there are: for $n=1$ the linear equation with well known formula for the general solution; for $n=2$ the Riccati equation well known both for theoretical \cite{Redheffer56Solutions,Redheffer57Riccati,Carinena07Riccati,Carinena11Riccati} and practical (see \cite{Carinena2011Lie} and references therein at the beginning of section 4) reasons; for $n=3$ the Abel differential equation of the first kind studied theoretically \cite{Mak2002New} and for practical reasons (\cite{Mak2001Solutions,Harko2003Relativistic,Xu2011Short} and references in \cite{Mak2002New,Carinena11Geometric}); for $n >3$ the generalized Abel differential equations \cite{Alwash2005Periodic, Alwash2007Periodic}.
Assuming $X_i = x^i \frac \partial {\partial x}$, for $i=0,\ldots,n$, are differential vector fields on $\RR$ one can, following Fliess \cite{Fliess81Fonctionnelles} (see also \cite{Kawski97NoncommutativePower, Gray2002Fliess} and \cite{Gray2008NonCausal} with references therein), expand the solution of the equation in terms of iterated integrals 
\begin{align*}
 \int_0^{t} \int_0^{t_k}\cdots \int_0^{t_2}u_{{i_k}}(t_k)\cdots u_{{i_1}}(t_1) \, dt_1\ldots dt_{k-1} dt_k
\end{align*}
and iterated differential operators $X_{i_1}\cdots X_{i_k}$ acting on the identity function $h(x) = x$ and evaluated at the zero point. In this approach one does not use a specific form of the vector fields, i.e. the fact that they are of polynomial type. In this article we show another approach to expanding a solution of the above equation in terms of iterated integrals with a use of an important feature of Chen's iterated integral mapping that it is a shuffle algebra homomorphism (see the comment after formula (\ref{eq:IteratedIntDef})). 
In fact with the use of Chen's mapping (\ref{eq:IteratedIntDef}) we will be able to consider a purely algebraic problem instead of the analytic one. More precisely, assuming that the solution of (\ref{eq:PreMain}) can be expanded in terms of iterated integrals we show that an algebraic equation must be satisfied in the space of non-commutative series on $n+1$ letters. It will be easy to show existence of the solution of the algebraic equation by a recursive formula of its homogeneous parts. The Chen's mapping gives us the expansion of the solution of initial problem as we state in Theorem \ref{thm:Analytic}. This is done in section \ref{sec:Existece}. Then, in section \ref{sec:CountingTrees}, by counting elements of a class of trees in two different ways we show convergence of the defined expansion of $x(t)$ for small times -- this is stated in Theorem \ref{thm:Main} (in section \ref{sec:Existece}).

As an application of this general approach we consider, in section \ref{sec:Examples},  the cases of the linear equation (i.e., $n=1$), the Riccati equation ($n=2$) and the one where there are only two non-vanishing summands. In the first case we reestablish a well known formula for the general solution, and in the second case we deduce  convergence of the series defining coordinates of the second kind connected with $\LieAlgebra{a}(1)$-type involutive distribution \cite{Pietrzkowski12Explicit}.

Finally, in section \ref{sec:Comparison}, we compare the Chen-Fliess approach with the one given in this article. It occurs that in the letter case the number of integrals to compute grows significantly slower with the order of approximation, than in the first case.

\section{Existence and convergence of an expansion}
\label{sec:Existece}

Let $n\in \NN$ \footnote{throughout the article we assume $\NN = \Set{0,1,2,3,\ldots}$ } , $T > 0$, and let $u_0,\ldots, u_n : [0,T] \to \RR$ be measurable and bounded (by a constant $M \in \RR$) functions. Consider a non-autonomous polynomial differential equation
\begin{align}
\label{eq:Main}
\dot x(t) &= u_0(t) + \NewtonSymbol n 1 u_1(t)x(t) +\cdots+ \NewtonSymbol{n}{i} u_i(t)x^i(t) + \cdots +  u_n(t) x^n(t), \\
\nonumber
x(0) &= 0.
\end{align}
Two comments are in order. Firstly, the Newton symbols occurring in the above formula are for convenience reasons -- without these constants it would be harder to estimate the radius of convergence of a defined series. Secondly, we assume the initial value equals zero. This is without loss of generality in a sense that with the linear change of variables $x \to x - x_0$ we can transform the equation with the initial value equals $x_0$ to another equation with different $u_i$'s.

Integrating both sides of the equation we get an integral equation
\begin{align}
\label{eq:MainIntegral}
x(t) = \int_0^{t} u_0(s) + \NewtonSymbol n 1 \, u_1(s)x(s) + \cdots + u_n(s) x^n(s)\, ds.
\end{align}
By Caratheory's theorem for a small $\epsilon> 0$ there exists an absolutely continues solution $\Solution : [0,\epsilon]\to\RR$ of the initial equation (\ref{eq:Main}) in a sense that $\Solution$ satisfies the integral equation (\ref{eq:MainIntegral}) for $t \in [0,\epsilon]$.  We want to express the solution $\Solution$ by means of iterated integrals of products of $u_i$'s. In order to do this we introduce some algebraic tools. 

To each function $u_i$ we assign a formal variable $a_i$, which we call a letter. The set of all letters $\Letters = \Set{a_0,\ldots,a_n}$ is called an alphabet. Juxtapositioning letters we can obtain words of an 
arbitrary length $k\in\NN$; the set of such words is denoted by $\WordsOfLength k = \SetSuchThat{b_1\cdots b_k}{b_1,\ldots, b_k \in \Letters}$. For $k=0$ the set $\WordsOfLength 0 = \Set{\EmptyWord}$ contains only one -- empty -- word. The set of all words is denoted by $\WordsOfLength{} = \bigcup_{k=0}^\infty \WordsOfLength{k}$.
The juxtaposition gives rise to an associative, noncommutative product on the set of words $\WordsOfLength{}\times\WordsOfLength{}\ni(v,w) \to v\cdot w = vw \in \WordsOfLength{}$ called the concatenation product; then the set $\WordsOfLength{}$ with the concatenation product and the neutral element $\EmptyWord\in\WordsOfLength{}$ is a free monoid generated by $\Letters$. Taking $\RR$-linear combination of words and 
bilinearly extending the concatenation product we get the $\RR$-algebras $\Polynomials{\RR}{\Letters}$ of noncommutative polynomials on $\Letters$ and $\Series{\RR}{\Letters}$ of noncommutative series on $\Letters$. 
In both algebras we can consider the bilinear product $\ShuffleProduct:\Series{\RR}{\Letters} \otimes\Series{\RR}{\Letters}\to \Series{\RR}{\Letters}$ -- the shuffle product --  defined recursively for words by $\EmptyWord\ShuffleProduct w = w\ShuffleProduct \EmptyWord = w$ for any $w\in\WordsOfLength{}$, and 
\begin{align}
\label{eq:ShuffleDef}
(vb)\ShuffleProduct(wc) = (v\ShuffleProduct(wc))\cdot b +  ((vb)\ShuffleProduct w)\cdot c
\end{align}
for all $b,c\in\Letters$ and $v, w\in\WordsOfLength{}$. 
It is easy to see that the shuffle product is commutative, thus with $\EmptyWord$ as the neutral element it gives rise to an additional commutative 
algebra structure on $\Polynomials{\RR}{\Letters}$ and $\Series{\RR}{\Letters}$. We will use both -- concatenation and shuffle -- products in our considerations.  It is important to indicate the priority of the shuffle product over the concatenation product in all formulas of this article, so that $v\ShuffleProduct w \cdot a$ always means $(v\ShuffleProduct w) \cdot a$.

On $\Series{\RR}{\Letters}$ we introduce a natural scalar product $\ScalarProduct{\cdot}{\cdot}{} : \Series{\RR}{\Letters} \times \Series{\RR}{\Letters} \to \RR$, which for elements $v,w \in \WordsOfLength{}\times\WordsOfLength{}$ is given by 
\begin{align*}
\ScalarProduct{v}{w}{} = 
\Cases{ 
\Case 1 {v=w} \\
\Case 0 {v\neq w}
}.
\end{align*}
For $S\in\Series{\RR}{\Letters}$, let $S_k \in \Polynomials{\RR}{\Letters}$ be the $k$-degree homogenous part of $S$, i.e., 
$$
S_k = \sum_{v\in\WordsOfLength{k}} \ScalarProduct{S}{v}{} \, v.
$$
Clearly, $S = \sum_{k=0}^\infty S_k$.

Define the linear homomorphism $\IteratedInt{t}:\Polynomials{\RR}{\Letters}\to \RR$ by $\IteratedInt t (\EmptyWord) = 1$, and
\begin{align*}
\WordsOfLength{k}\ni v = a_{i_1}\cdots a_{i_k} \mapsto \IteratedInt{t}(v) := \int_0^{t}u_{{i_k}}(t_k) \int_0^{t_k}\cdots \int_0^{t_2}u_{{i_1}}(t_1) \, dt_1\ldots dt_{k-1} dt_k.
\end{align*}
Equivalently, the homomorphism can be defined recursively by 
\begin{align}
\label{eq:IteratedIntDef}
\IteratedInt t(va_i) := \int_0^t  \IteratedInt s(v)\, u_i(s)\, ds
\end{align}
for any $ v\in\WordsOfLength{}$ and $a_{i}\in\Letters$. Since $u_i$'s are bounded the definition is correct for all $t \geq 0$.
One can check that $\IteratedInt t$ is in fact a shuffle algebra homomorphism, i.e., $\IteratedInt t (v\ShuffleProduct w) = \Upsilon^t(v)\,\Upsilon^t(w)$ (see \cite{Chen68Algebraic,Reutenauer93FreeLie}) which is a crucial feature in what follows. For a general series $S\in\Series{\RR}{\Letters}$ the homomorphism $\IteratedInt t$ is obviously not well defined since $\IteratedInt t (S)$ can be divergent. We restrict the definition of $\IteratedInt t $ to series $S\in\Series{\RR}{\Letters}$ and times $t \geq 0$ for which the series 
$$\sum_{k=0}^\infty \IteratedInt t (S_k)$$ is convergent.

Coming back to the initial problem, assume that there exists a series $\MainSeries\in\Series{\RR}{\Letters}$ such that the solution $\Solution$ of (\ref{eq:Main}) satisfies $\Solution(t) = \IteratedInt{t}(\MainSeries)$ for $t\in[0,\epsilon]$ (in particular $\IteratedInt{t}(\MainSeries)$ is convergent). Using the recursive definition of $\IteratedInt{t}$ (\ref{eq:IteratedIntDef}) and the fact that $\IteratedInt{t}$ is a shuffle algebra homomorphism, we get from (\ref{eq:MainIntegral}) that
\begin{align*}
\IteratedInt{t}(\MainSeries) = \IteratedInt{t}(a_0 + \NS n 1 \, \MainSeries\cdot a_1 + \NS n 2\, \MainSeries\ShuffleProduct\MainSeries\cdot a_2 + \ldots + \MainSeries^{\ShuffleProduct n}\cdot a_n),
\end{align*}
where we abbreviate $\NS n i = \NewtonSymbol n i$, and $\MainSeries^{\ShuffleProduct n}$ is defined recursively in a natural way, i.e., $\MainSeries^{\ShuffleProduct 0} = \EmptyWord$ and $\MainSeries^{\ShuffleProduct n} = \MainSeries\ShuffleProduct \MainSeries^{\ShuffleProduct (n-1)}$.
Now the point is that we can forget, for a moment, about the homomorphism and consider only the algebraic equation
\begin{align}
\label{eq:MainSeries}
\MainSeries = a_0 + \NS n 1\, \MainSeries\cdot a_1 + \NS n 2\, \MainSeries\ShuffleProduct\MainSeries\cdot a_2 + \ldots + \MainSeries^{\ShuffleProduct n}\cdot a_n.
\end{align}

\begin{proposition}
\label{prop:Existence}
There exists the unique solution $\MainSeries\in\Series{\RR}{\Letters}$ of the algebraic equation (\ref{eq:MainSeries}).
\end{proposition}

\begin{proof}
The equation under consideration must be satisfied for each homogeneous part, so we can split it into the following series of equations
\begin{align*}
\MainSeries_0 &= 0, &
\MainSeries_1 &= a_0, \\
\MainSeries_2 &= \NS n 1 \, \MainSeries_1\cdot a_1, &
\MainSeries_3 &= \NS n 1 \, \MainSeries_2\cdot a_1 + \NS n 2 \, \MainSeries_1\ShuffleProduct\MainSeries_1\cdot a_2, \\
\end{align*}
and for arbitrary $k\in\NN$
\begin{align}
\label{eq:Recurence}
\MainSeries_{k+1} &=  \sum_{i=1}^n \NewtonSymbol n i  \sum_{\MultiIndex{l}\in\MultiIndexSet{i}} \MainSeries_{l_1}\ShuffleProduct\cdots\ShuffleProduct\MainSeries_{l_i} \cdot a_i,
\end{align}
where the second sum is taken over multi-indices $\MultiIndex{l} = (l_1,\ldots,l_i)$ in
$$
\MultiIndexSet{i}= \SetSuchThat{(l_1,\ldots,l_i)\in\NN^i}{l_1,\ldots,l_i \geq 1,\, l_1 + \cdots + l_i = k }.
$$
We see that the homogeneous parts of the series $\MainSeries$ are defined recursively, therefore the series is defined uniquely. 
\end{proof}

Observe that from the recursive definition of $\IteratedInt t$ and a property $\IteratedInt t (v\ShuffleProduct w) = \IteratedInt t (v) \, \IteratedInt t (w)$ we get
 \begin{align}
 \label{eq:SolutionKOne}
  \Solution_{k+1}(t) &=  \sum_{i=1}^n \NewtonSymbol n i  \sum_{\MultiIndex{l}\in\MultiIndexSet{i}} \int_0^t \Solution_{l_1}(s)\cdots\Solution_{l_i}(s) \, u_i(s)\, ds,
 \end{align}
where we use an abbreviation $\Solution_k(t) = \IteratedInt t (\MainSeries_k)$. By this definition $\Solution(t) = \sum_{k=0}^\infty \Solution_k (t)$. Moreover, any permutation of $(l_1,\ldots,l_i)$ gives the same expression under the integral. For $\MultiIndex{l} \in \MultiIndexSet{i}$ denote by $R(\MultiIndex{l})$ the number of such permutations, i.e., 
$$
R(\MultiIndex{l}) = \Cardinality\SetSuchThat{\sigma \in \Sigma_i}{l_1 = l_{\sigma (1)},\ldots, l_i = l_{\sigma (i)}}.
$$
Then
\begin{align}
\label{eq:SolutionKTwo}
  \Solution_{k+1}(t) &=  \sum_{i=1}^n \NewtonSymbol n i  \sum_{\MultiIndex{l}\in\MultiIndexSetOrdered{i}} R(\MultiIndex{l}) \int_0^t \Solution_{l_1}(s)\cdots\Solution_{l_i}(s) \, u_i(s)\, ds,
 \end{align}
where the second sum is taken over
$$
\MultiIndexSetOrdered{i}= \SetSuchThat{(l_1,\ldots,l_i)\in\NN^i}{1 \leq l_1 \leq l_2 \leq \cdots \leq l_i,\, l_1 + \cdots + l_i = k }.
$$
 
Let us state it in the following theorem.

\begin{theorem}
\label{thm:Analytic}
 Let $\Solution_1(t) = \int_0^t u_0(s) \, ds$ and recursively for $k \geq 1$ $\Solution_{k+1}(t)$ is given by \eqref{eq:SolutionKOne} or \eqref{eq:SolutionKTwo}.
 Then $\Solution(t) = \sum_{k=1}^\infty \Solution_k (t)$ is a formal solution of the differential equation (\ref{eq:Main}).
\end{theorem}

\begin{remark}
\label{rem:Growth}
 It is worth noticing that for a fixed $k \geq 1$, the number of integrals in formula \eqref{eq:SolutionKTwo} is the cardinality of $\sum_{i=1}^n \MultiIndexSetOrdered{i}$. This is less than the cardinality of $\sum_{i=1}^\infty \MultiIndexSetOrdered{i}$, which is the number of partitions of $k$. The first ten of these numbers are 1, 2, 3, 5, 7, 11, 15, 22, 30, 42. It means that the number of integrals that we have to perform to compute $\Solution_{k+1}$ grows quite slowly. In section \ref{sec:Comparison} we show that this growth is much slower than the growth of the number of non-zero components in the Chen-Fliess expansion. 
\end{remark}

 There remains the problem under what assumption the solution for the algebraic equation is in the domain of the homomorphism $\IteratedInt t : \Series{\RR}{\Letters} \supset\Domain{\IteratedInt t } \to \RR$, i.e., when $\sum_k\IteratedInt{t} (\MainSeries_k)$ is convergent.
In order to solve it, we need to compute the number of words (with multiplicities) in each homogeneous part of $\MainSeries$. So for $S\in\Polynomials \RR \Letters$ let us introduce the following definition
\begin{align*}
\NumberOfWords S  = \sum_{v\in\WordsOfLength{}} |\ScalarProduct{S}{v}{}|.
\end{align*}

\begin{proposition}
\label{prop:NumberOfWords}
If $\MainSeries_k$ is the $k$-degree homogeneous part of the solution $\MainSeries$ of the algebraic equation (\ref{eq:MainSeries}), then for $k \geq 1$, 
$
\NumberOfWords\MainSeries_k = ((n-1)(k-1) +1)\cdot((n-1)(k-2) +1)\cdots n
$ 
(in particular $\NumberOfWords\MainSeries_1 = 1$) and $\NumberOfWords\MainSeries_0 = 0$.
\end{proposition}
In particular, for $n=0$, $\NumberOfWords\MainSeries_1 = 1$ and $\NumberOfWords\MainSeries_k = 0$ otherwise; for $n = 1$, $\NumberOfWords\MainSeries_k = 1$; for $n = 2$, $\NumberOfWords\MainSeries_k = k!$; for $n = 3$, $\NumberOfWords\MainSeries_k = (2k-1)!!$;  for $n = 4$, $\NumberOfWords\MainSeries_k = (3k-2)!!!$, and so on.   

The proposition will be proved in section \ref{sec:CountingTrees}.

Now we state the theorem about convergence of the expansion.
\begin{theorem}
\label{thm:Main}
Let $n\in\NN$ and assume $u_0,\ldots, u_n:[0,T]\to\RR$ are measurable functions such that $|u_i| \leq M$ for an $M>0$. Let $\MainSeries\in\Series{\RR}{\Letters}$ be the unique solution of the algebraic equation (\ref{eq:MainSeries}). Then the series $\sum_k \Solution_k(t) = \sum_k \IteratedInt t (\MainSeries_k)$ is absolutely convergent for $0 \leq t < \min\Set{T, 1/(M(n-1))}$ if $n\geq 2$,  $0\leq t \leq T$ if $n=0,1$, and $\Solution(t) = \IteratedInt t (\MainSeries)$ is the solution of the differential equation (\ref{eq:Main}) on the same segment.
\end{theorem}

\begin{proof}
For $v\in\WordsOfLength k $ the iterated integral $\IteratedInt t (v)$ is in fact taken over a $k$-dimensional simplex of $k$-dimensional measure $t^k/k!$. Since $u_i$'s are bounded by $M$ we have $|\IteratedInt t (v)| \leq (M t)^k/k!$. Therefore,
\begin{align*}
|\IteratedInt t (\MainSeries_k)| \leq \NumberOfWords\MainSeries_k\, \frac{(M t)^k}{k!} = \frac{ ((n-1)(k-1) +1)\cdot((n-1)(k-2) +1)\cdots n}{k!}\, {(Mt)^k}.
\end{align*}
Since $\NumberOfWords\MainSeries_{k+1}/\NumberOfWords\MainSeries_k = \frac {(n-1)k+1}{k+1} \xrightarrow[k\to\infty]{} n-1$, the series $\sum_k \IteratedInt t (\MainSeries_k)$ is convergent for $t < 1/(M(n-1))$ if $n \geq 2$, and $t < T$ if $n=1$. For $n=0$ the statement is obvious.
\end{proof}

\section{Counting trees}
\label{sec:CountingTrees}

In this section we prove Proposition \ref{prop:NumberOfWords}. In order to do this we consider certain classes of trees. It occurs that the number of trees in these particular classes equal $\NumberOfWords \MainSeries_k$ on the one hand and $((n-1)(k-1) +1)\cdot((n-1)(k-2) +1)\cdots n$ on the other hand.

For $k,n\in\NN$ let $\Trees{n}{k}$ denote the set of planar, rooted, full $n$-ary and increasingly labeled
trees on $k$  vertices. Recall that a tree is rooted if there exists a distinguished vertex called the root; is full $n$-ary if each vertex has exactly none or $n$ children; is on $k$ vertices if the number of 
vertices with $n$ children (parent vertices) is equal $k$; is increasingly labeled if the parent vertices 
are labeled by  natural numbers from 1 to $k$, and the labels increase along each branch starting at the root 
(in particular the root is labeled by "$1$"). A leaf of a tree is a non-parent vertex, i.e., a vertex 
without children. It is important to note that the number of leafs in each tree in $\Trees{n}{k}$ is 
constant and equals $(n-1)k +1$. Indeed, using induction on $k$ we see that for $k=0$ the only tree in $
\Trees{n}{0}$ has $0$ children, so the root is the only leaf; each tree $\Trees{n}{k}$ can be obtained from a tree $\Tree
\in\Trees{n}{k-1}$ by adding $n$ leafs to a certain leaf of $\Tree$, so 
the number of leafs increases by $(n-1)$.

Now we count the cardinality of $\Trees{n}{k}$ in two different ways.

\begin{lemma}
\label{lem:Cardinality}
The cardinality of $\Trees{n}{k}$ equals $\Cardinality{\Trees{n}{k}} = ((n-1)(k-1) +1)\cdot((n-1)(k-2) +1)\cdots n$ for $k\geq 1$ and $\Trees{n}{0}=1$.
\end{lemma}
\begin{proof} The case $n=0$ is trivial.
Fix $n\in\NN$ s.t. $n\geq 1$. We proceed by induction on $k\in\NN$. For $k=0$ there is only one tree, so the statement is correct. Assume $\Cardinality{\Trees{n}{k} = ((n-1)(k-1) +1)\cdot((n-1)(k-2) +1)\cdots n}$. Each tree in $\Trees{n}{k+1}$ comes from the unique tree $\Tree$ in $\Trees{n}{k}$ by adding label "$k+1$" and $k$ vertices to a leaf of $\Tree$. Since the number of leafs of $\Tree$ is equal $(n-1)k +1$ we obtain the result. 
\end{proof}

\begin{lemma}
\label{lem:Recurrence}
For $k\in\NN $ the cardinality of $\Trees{n}{k+1}$ equals 
\begin{align*}
\Cardinality{\Trees{n}{k+1}} &=  \sum_{\MultiIndex{l}\in\MultiIndexSetZero{n}} \NewtonSymbolGeneral{k}{l}{n}\, \Cardinality{\Trees{n}{l_1}}\cdots \Cardinality{\Trees{n}{l_n}},
\end{align*}
where the sum is taken over multi-indices $\MultiIndex{l} = (l_1,\ldots,l_n)$ in
$$
\MultiIndexSetZero{n}= \SetSuchThat{(l_1,\ldots,l_n)\in\NN^n}{l_1 + \cdots + l_n = k }.
$$
\end{lemma}

\begin{proof}
First of all observe that for $k\in\NN$ each tree in $\Trees{n}{k+1}$  is uniquely given by $n$ trees $\Tree_1\in\Trees{n}{l_1},\ldots,\Tree_n\in\Trees{n}{l_n}$ such that $l_1 + \cdots + l_n = k$, and a partition of the set $\Set{2,\ldots,k+1}$ into $n$ disjoint sets $I_1,\ldots, I_n$ of the cardinality $\Cardinality I_i = l_i$ for $i=1,\ldots,n$ (we do not assume $I_i\neq \emptyset$), i.e.,
\begin{align*}
 \Trees{n}{k+1} \sim \bigsqcup_{\MultiIndex{l}\in\MultiIndexSetZero{n}} {\Trees{n}{l_1}}\times\cdots \times{\Trees{n}{l_n}} \times I(\MultiIndex{l}),
\end{align*}
where $I(\MultiIndex{l})$ is the set of all partitions of the set $\Set{2,\ldots,k+1}$ into $n$ disjoint sets $I_1,\ldots, I_n$ s.t.  $\Cardinality I_i = l_i$ for $i=1,\ldots,n$.
Indeed, the 
root of a given tree $\Tree \in \Trees{n}{k+1}$ has $n$ child vertices $\Vertex_1,\ldots,\Vertex_n$. Each $\Vertex_i$ is the root of a certain maximal subtree $\tilde\Tree_i$ of $\Tree$. We assume that $\tilde\Tree_i$ has $l_i$ parent vertices, which are labeled by some numbers $2 \leq a_i^1 < \cdots < a_i^{l_i} \leq k+1$. Obviously, $l_1 + \cdots + l_n = k$. Changing the label "$a^j_i$" into a label "$j$" we obtain a tree $\Tree_i\in\Trees{n}{l_i}$. 
Defining $I_i = \Set{a_i^1, \ldots , a_i^{l_i}}$ for $i=1,\ldots,n$, we have a partition of $\Set{2,\ldots,k+1}$ into $n$ disjoint sets, i.e., $\Set{2,\ldots,
k+1} = I_1\cup\cdots\cup I_n$. It is clear how to invert this procedure in order to get its uniqueness.

Using the above correspondence it is easy to establish the formula in the lemma since there are $\NewtonSymbolGeneral{k}{l}{n}$ possible partitions of the set $\Set{2,\ldots,k+1}$  into $n$ disjoint parts $I_1,\ldots,I_n$ such that $\Cardinality I_i = l_i \in \NN$, i.e., $\Cardinality I(\MultiIndex{l}) = \NewtonSymbolGeneral{k}{l}{n}$.
\end{proof}

We are now ready to prove Proposition \ref{prop:NumberOfWords}.

\begin{proof}[Proof of Proposition \ref{prop:NumberOfWords}]
First of all observe that for $i\in\NN$, $l_1,\ldots,l_i\in\NN$, and words $v_1\in\WordsOfLength{l_1},\ldots,v_i\in\WordsOfLength{l_i}$, the number of words in the shuffle product $v_1\ShuffleProduct\cdots\ShuffleProduct v_i$ equals 
$$\NumberOfWords(v_1\ShuffleProduct\cdots\ShuffleProduct v_i) = \NewtonSymbolGeneral{(l_1 + \cdots + l_i)}{l}{i}.$$
Using this fact, homogeneity of polynomials $\MainSeries_l$, and the recursive formula (\ref{eq:Recurence}) we get
\begin{align}
\label{eq:One}
\NumberOfWords\MainSeries_{k+1} &=  \sum_{i=1}^n \NewtonSymbol n i  \sum_{\MultiIndex{l}\in\MultiIndexSet{i}} \NumberOfWords\left(\MainSeries_{l_1}\ShuffleProduct\cdots\ShuffleProduct\MainSeries_{l_i}\right) \\ \nonumber
&=
\sum_{i=1}^n \NewtonSymbol n i  \sum_{\MultiIndex{l}\in\MultiIndexSet{i}} 
\NewtonSymbolGeneral{(l_1 + \cdots + l_i)}{l}{i}\cdot
\NumberOfWords\MainSeries_{l_1}\cdots\NumberOfWords\MainSeries_{l_i}
\end{align}
where $\MultiIndexSet{i}$ contains multi-indices $(l_1,\ldots,l_i)\in\NN^i$ such that ${l_1 + \cdots + l_i = k }$ and, what is  important, $l_1,\ldots,l_i \geq 1$. In order to get rid of the first sum, we introduce the  following notation
\begin{align*}
\Number_k = 
\Cases{
\Case{1}{k=0} \\
\Case{\NumberOfWords\MainSeries_k}{k\neq 0}
}
\end{align*} 
and allow $l_1,\ldots,l_i$ to be equal $0$. Then we rewrite (\ref{eq:One}) as
\begin{align}
\label{eq:Two}
\Number_{k+1} = 
\sum_{\MultiIndex{l}\in\MultiIndexSetZero{n}} 
\NewtonSymbolGeneral{(l_1 + \cdots + l_n)}{l}{n}\cdot
\Number_{l_1}\cdots\Number_{l_n},
\end{align}
where $\MultiIndexSetZero{n}= \SetSuchThat{(l_1,\ldots,l_n)\in\NN^n}{ l_1 + \cdots + l_n = k }.$
Indeed, if $l_1,\ldots,l_n \in \NN$ and only $i$ of them, say $\hat l_1,\ldots,\hat l_i$, are not equal $0$, then 
\begin{align*}
\NewtonSymbolGeneral{(l_1 + \cdots + l_n)}{l}{n}\cdot
\Number_{l_1}\cdots\Number_{l_n} = \NewtonSymbolGeneral{(\hat l_1 + \cdots + \hat l_i)}{\hat l}{i}\cdot
\Number_{\hat l_1}\cdots\Number_{\hat l_i}.
\end{align*}
Clearly, there are $\NewtonSymbol{n}{i}$ different multi-indices $(l_1,\ldots,l_n)$ satisfying this condition, and this is the reason for the Newton symbol to disappear in formula (\ref{eq:Two}).  

Finally, we see that by Lemma \ref{lem:Recurrence} the recursive formula (\ref{eq:Two}) for the numbers $\Number_k$ overlaps with the one for the cardinality of trees $\Cardinality\Trees{n}{k}$. Since, the series coincide for $k = 0$, i.e., $\Number_0 = \Cardinality\Trees{n}{0} = 1$, we conclude using Lemma \ref{lem:Cardinality} that 
$$
\NumberOfWords\MainSeries_k = \Number_k = ((n-1)(k-1) +1)\cdot((n-1)(k-2) +1)\cdots n
$$
for $k \geq 1$. The fact that $\NumberOfWords\MainSeries_0 = 0$ is trivial.
\end{proof}

\begin{remark}
The above proof can be simplified for $n = 0, 1$ when $\NumberOfWords\MainSeries_k \leq 1$, but also for $n = 2$. In this case the recursive formula (\ref{eq:Recurence}) gives
\begin{align*}
\NumberOfWords\MainSeries_{k+1} &= 2\,\NumberOfWords\MainSeries_{k} + \sum_{j=1}^{k-1} \NewtonSymbol{k}{j} \NumberOfWords\MainSeries_{j}\,\NumberOfWords\MainSeries_{k-j}.
\end{align*}
Assuming the inductive hypothesis $\NumberOfWords\MainSeries_l = l!$ for $l\leq k$ we obtain
\begin{align*}
\NumberOfWords\MainSeries_{k+1} &= 2\, k! + \sum_{j=1}^{k-1} \NewtonSymbol{k}{j} j! (k-j)! = 2\, k! + (k-1)\, k! = (k+1)!\, .
\end{align*}
\end{remark}

\section{Examples}
\label{sec:Examples}

In this section we discuss the three simplest cases when $n=0,1,2$, and the case where only $u_0$ and $u_n$ are not vanishing.
We will need one additional intuitive notation. Namely, for $S\in\Series{\RR}{\Letters}$ such that $\ScalarProduct{S}{\EmptyWord}{} = 0$ we define the shuffle exponent
\begin{align*}
\ExpShuffle{S} = \sum_{k=0}^\infty \frac{S^{\ShuffleProduct k}}{k!},
\end{align*}
where we recall that $S^{\ShuffleProduct 0} = \EmptyWord$ and $S^{\ShuffleProduct k} = S\ShuffleProduct S^{\ShuffleProduct (k-1)}$.

If $n=0$, then the equation (\ref{eq:Main}) is $\dot x(t) = u_0(t), x(0) = 0$ and obviously a solution is $\Solution(t) = \IteratedInt{t}(\MainSeries)$, where $\MainSeries = \MainSeries_1 = a_0$ is homogeneous of degree 1.

Let us pass to the case $n=1$ when (\ref{eq:Main}) is a linear equation $\dot x(t) = u_0(t) + u_1(t) x(t)$, $x(0) = 0$, which can be solved by variation of parameter. Let us see how it can be done using the series $\MainSeries$. Using recursive formula (\ref{eq:Recurence}) 
\begin{align*}
\MainSeries_0 &= 0, & \MainSeries_1 & = a_0, &  \MainSeries_{k+1} &= \MainSeries_k\cdot a_1,
\end{align*}
we get $\MainSeries = a_0\cdot(\EmptyWord +  a_1 + a_1^2 + a_1^3 + \cdots).$ If we use formula $a_1^k = a_1^{\ShuffleProduct k}/k!$, we get
\begin{align}
\label{eq:MainSeriesNOne}
\MainSeries = a_0\cdot \ExpShuffle{a_1}.
\end{align}
This expression looks nice, but there is a problem since $a_0$ factor is on the left hand side and therefore the expression will not simplify if we apply  $\IteratedInt{t}$ to it. 
In order to obtain the solution in a common form, we prove the following lemma.

\begin{lemma}
For $a_0, a_1 \in \Letters$ it follows that 
$$\sum_{k=0}^{\infty} a_0a_1^k = \sum_{k,l=0}^\infty (-1)^l\, a_1^k\ShuffleProduct(a_1^l a_0).$$
\end{lemma}

\begin{proof}
Observe first that for $k,l\in\NN$ 
\begin{align}
\label{eq:Three}
a_1^k\ShuffleProduct(a_1^l a_0) = a_1^la_0a_1^k + \NewtonSymbol{l+1}{1}a_1^{l+1} a_0 a_1^{k-1}+\cdots+ \NewtonSymbol{l+k}{k}a_1^{l+k} a_0.
\end{align}
Indeed, for $k=0$ the formula is correct. Using the inductive hypothesis for each $m \leq k$, and the defining formula for shuffle product (\ref{eq:ShuffleDef})  we get
\begin{align*}
a_1^{k+1}\ShuffleProduct(a_1^l a_0) &= (a_1^{k}\ShuffleProduct(a_1^l a_0))a_1 + (a_1^{k+1}\ShuffleProduct a_1^l)a_0 \\
&= a_1^la_0a_1^{k+1} + \NewtonSymbol{l+1}{1}a_1^{l+1} a_0 a_1^{k} +\cdots+ \NewtonSymbol{l+k}{k}a_1^{l+k} a_0 a_1 + (a_1^{k+1}\ShuffleProduct a_1^l)a_0.
\end{align*}
Since  $a_1^{k+1}\ShuffleProduct a_1^l = \NewtonSymbol{l+k+1}{k+1}a_1^{l+k+1}$, 
we obtain formula (\ref{eq:Three}).

Using the above proved formula we see that
\begin{align*}
\sum_{k,l=0}^\infty (-1)^l\, a_1^k\ShuffleProduct(a_1^l a_0) &=
\sum_{k,l=0}^\infty \sum_{m=0}^{k} (-1)^l\,  \NewtonSymbol{l+m}{m}a_1^{l+m} a_0 a_1^{k-m} \\
&=
\sum_{k',l'=0}^\infty \left[\sum_{m=0}^{l'} (-1)^{l'-m}\,  \NewtonSymbol{l'}{m}\right]\, a_1^{l'} a_0 a_1^{k'},
\end{align*} 
where in the last line we change a method of summation by putting $k'=k-m$ and $l'=l+m$. Since the expression in the squared brackets equals $0^{l'}$, the sum over $l'$ reduces to the one summand with $l' = 0$, and therefore
\begin{align*}
\sum_{k,l=0}^\infty (-1)^l\, a_1^k\ShuffleProduct(a_1^l a_0) &=
\sum_{k'=0}^\infty  a_0 a_1^{k'}.
\end{align*} 
This ends the proof.
\end{proof}

From the lemma it follows that 
\begin{align*}
\MainSeries &= \sum_{k=0}^{\infty} a_0a_1^k = \sum_{k,l=0}^\infty (-1)^l\, a_1^k\ShuffleProduct(a_1^l a_0) 
= \sum_{k=0}^\infty \frac {a_1^{\ShuffleProduct k}}{k!} \ShuffleProduct \left(\sum_{k=0}^\infty  \frac{(-a_1)^{\ShuffleProduct l}}{l!}\cdot a_0\right) \\
&= \ExpShuffle{a_1}\ShuffleProduct(\ExpShuffle{-a_1}\cdot a_0).
\end{align*}
Since $\IteratedInt{t}:\Series{\RR}{\Letters}\to \RR$ is a shuffle-algebra homomorphism, it follows that $\IteratedInt{t}(\ExpShuffle{S}) = \exp(\IteratedInt{t}(S))$ for all series $S$ such that $\ScalarProduct{S}{\EmptyWord}{} = 0$, and therefore
\begin{align*}
\Solution(t) = \IteratedInt{t}(\MainSeries) = \exp\left(\int_0^t u_1(s)\, ds\right)\, \int_0^t \exp\left(-\int_0^s u_1(\tau)\, d\tau\right)\, u_0(s)\, ds,
\end{align*}
which is the standard formula.

For $n=2$ the equation under consideration is  
\begin{align}
\label{eq:Riccati}
\dot x(t) &= u_0(t) + 2 u_1(t) x(t) + u_2(t) x^2(t), & x(0) &= 0,
\end{align} 
that is a general Riccati equation. 
In this case, the series $\MainSeries$ is the unique solution of
\begin{align}
\label{eq:MainSeriesTwo}
\MainSeries = a_0 + 2\, \MainSeries\cdot a_1 +  \MainSeries\ShuffleProduct\MainSeries\cdot a_2,
\end{align}
and therefore $\MainSeries = \sum_k \MainSeries_k$, where $\MainSeries_k$ are given by the recurrence 
\begin{align}
\label{eq:RecurenceTwo}
\MainSeries_0 &= 0, & \MainSeries_1 & = a_0, &  \MainSeries_{2} &= 2\, \MainSeries_1\cdot a_1, & 
\MainSeries_{k+1} &= 2\, \MainSeries_k\cdot a_1 +  \sum_{{l}=1}^{k-1} \MainSeries_{l}\ShuffleProduct\MainSeries_{k-l} \cdot a_2.
\end{align}
Let us mention that the Riccati equation is a Lie-Scheffers system of the type $\mathfrak{a}_1$ (see \cite{Carinena07Riccati,Carinena11Riccati} and \cite{Redheffer56Solutions,Redheffer57Riccati}). More precisely, if we take vector fields 
\begin{align*}
 X_0(x) & = \frac \partial {\partial x}, & X_1(x) &= 2x \frac \partial {\partial x}, &  X_2(x) &= x^2 \frac \partial {\partial x}
\end{align*}
on $\RR$, then they satisfy the following commutation relations
\begin{align*}
[X_0,X_1] &= 2X_0, & [X_0, X_2] &= X_1, & [X_1, X_2] = 2 X_2.
\end{align*}
It means the vector fields span a simple Lie algebra of the type $\mathfrak{a}_1$ (isomorphic to $\mathfrak{sl}(2,\RR)$), and thus \eqref{eq:Riccati} -- equivalent to $\dot x(t) = \sum u_i(t) X_i$ -- is a Lie-Scheffers system of this type. The solution in terms of iterated integrals of $u_i$'s for such a system was given in \cite{Pietrzkowski12Explicit}. Let us recall the main theorems of this article.

\begin{theorem}[Theorem 1 in \cite{Pietrzkowski12Explicit}]
\label{thm:Pietrzkowski12Explicit}
Let $X_a, X_b, X_c \in \TangentFields{\Manifold}$ be smooth tangent vector fields on a manifold $\Manifold$ satisfying $[X_a,X_b] = 2X_a,\,  [X_a, X_c] = -X_b,\,  [X_b, X_c] = 2 X_c$. Let $u_a, u_b, u_c : [0,T]\to \RR$ be fixed measurable functions. Then (locally)
the solution $x:[0,T]\to\Manifold$ of the differential equation 
\begin{align*}
 \dot x(t)  & = u_c(t) X_c + u_b(t) X_b + u_a(t) X_a, & x(0) &= x_0 \in \Manifold
\end{align*}
is of the form
\begin{align}
\label{eq:Solution}
x(t) = \VectorFlow{\Xi_c(t) X_c} \VectorFlow{\Xi_b(t) X_b} \VectorFlow{\Xi_a(t) X_a} (x_0).
\end{align}
Here, $\Xi_a, \Xi_b, \Xi_c : [0,T] \to \RR$ are given by $\Xi_d(t) := \Upsilon^t(\WordsInSeriesOfLength{d}{})$ (for $d = a, b, c$), where 
\begin{align}
\label{eq:ThreeFormulas}
 \WordsInSeriesOfLength{a}{} &= a\cdot \ExpShuffle{2\AOneSeries}, &
\WordsInSeriesOfLength{b}{} &= {\AOneSeries}, &
\WordsInSeriesOfLength{c}{} &= \ExpShuffle{2\AOneSeries}\cdot c, 
\end{align}
and $\AOneSeries\in\Series{\RR}{\Letters}$  
is the unique solution of the algebraic equation
\begin{align*}
\AOneSeries =  b - a\cdot \ExpShuffle{2\AOneSeries}\cdot c.
\end{align*}
In particular, we have
\begin{align*}
b - a\cdot\WordsInSeriesOfLength{c}{} = \WordsInSeriesOfLength{b}{} = b - \WordsInSeriesOfLength{a}{}\cdot c.
\end{align*}
\end{theorem}

\begin{theorem}[Theorem 2 in \cite{Pietrzkowski12Explicit}]
\label{thm:Riccati}
For fixed measurable functions $u_a, u_b, u_c:[0,T]\to\RR$ the function $\Xi_a:[0,T]\to\RR$, defined in Theorem \ref{thm:Pietrzkowski12Explicit} by $\Xi_a(t) = \Upsilon^t(a\cdot \ExpShuffle{2\AOneSeries})$, is (locally) the solution of the Riccati equation:
\begin{align*}
\dot\Xi_a(t) &= u_a(t) + 2u_b(t)\, \Xi_a(t) - u_c(t)\, \Xi_a^2(t), &
\Xi_a(0) &= 0.
\end{align*}
\end{theorem} 

Observe that taking $X_a = X_0$, $X_b = X_1$, $X_c = - X_2$ (and therefore $c = -a_2$), and $u_a = u_0$, $u_b = u_1$, $u_c = u_2$, and $\Xi_a(t) = x(t)$ the system \eqref{eq:Riccati} can be put into the context of the above theorems in the following way. From Theorem \ref{thm:Riccati} we conclude that the solution of \eqref{eq:Riccati} is $x(t) = \IteratedInt t (\MainSeries)$, where
\begin{align*}
\MainSeries = a_0\cdot \ExpShuffle{2\AOneSeries},
\end{align*} 
and $\AOneSeries\in\Series{\RR}{\Letters}$  
is, by Theorem \ref{thm:Pietrzkowski12Explicit}, the unique solution of the algebraic equation
\begin{align}
\label{eq:AOneSeries}
\AOneSeries =  a_1 + a_0\cdot \ExpShuffle{2\AOneSeries}\cdot a_2.
\end{align}
Additionally, from the last line in Theorem \ref{thm:Pietrzkowski12Explicit} we conclude that
\begin{align}
\label{eq:AOneSeriesAdd}
\AOneSeries = a_1 + \MainSeries\cdot a_2.
\end{align}
In the discussed article, there was also given a recursive formula for $\AOneSeries$, but observe that in fact the algebraic equation (\ref{eq:MainSeriesTwo}) for $\MainSeries$ is simpler than the equation (\ref{eq:AOneSeries}) for $\AOneSeries$. In consequence, it is reasonable to invert this statement saying that the 
series $\AOneSeries$ is given by \eqref{eq:AOneSeriesAdd}, where $\MainSeries$ is the solution of (\ref{eq:MainSeriesTwo}). Now using Theorem \ref{thm:Main} we get the following corollary about the $\LieAlgebra{a}_1$-type Lie-Scheffers system considered in Theorem \ref{thm:Pietrzkowski12Explicit}.

\begin{proposition}
In the context of Theorem \ref{thm:Pietrzkowski12Explicit}, if $|u_d| \leq M$ for an $M>0$ ($d = a,b,c$), then the solution \eqref{eq:Solution} exists for $0 \leq t < \min\Set{T, 1/M}$.
\end{proposition}

\begin{proof}
By Theorem \ref{thm:Analytic} the function $\IteratedInt t (\MainSeries)$ is well defined for $0 \leq t < \min\Set{T, 1/M}$.
The above observations (in particular formula \eqref{eq:AOneSeriesAdd}) implies that $\IteratedInt t (\AOneSeries)$ is also well defined in this interval. Finally, by formulas \eqref{eq:ThreeFormulas} each function $\Xi_d(t) := \IteratedInt t (\WordsInSeriesOfLength{d}{})$ ($d=a, b, c$) is defined for $0 \leq t < \min\Set{T, 1/M}$, too.
\end{proof}

Let us observe that in each of the discussed cases the solution is of the form $\MainSeries = a_0\cdot\ExpShuffle{L}$, where $L\in\Series{\RR}{\Letters}$ such that $\ScalarProduct{L}{\EmptyWord}{} = 0$, and the series $L$ in case $n$ reduces to $L$ in case $n-1$ if taking $u_n \equiv 0$. Indeed, $L = 0$ for $n=0$, $L = a_1$ for $n=1$, and $L = \AOneSeries$ for $n=2$ reduces, by (\ref{eq:AOneSeries}), to $a_1$ for $u_2 \equiv 0$. This observation suggests a question: if the same holds for all $n\in\NN$? Since Riccati 
equation is essentially the only differential equation on a real line which is connected with the action of a group (namely the special linear group $SL(2)$) \cite{Carinena99Integrability, Carinena2011Lie} one could anticipate that a generalization is impossible. Nevertheless the problem is open.

Another example we are going to consider is the one where there are only two non-vanishing summands, i.e.,  
$\dot x(t) = \NewtonSymbol n m u(t)x^m(t) + u_n(t) x^n(t)$, $x(0) = 0$, and $0\leq m < n$ are fixed. The case $m\neq 0$ has the trivial solution $x(t) \equiv 0$, so in fact we consider 
\begin{align}
\label{eq:n}
\dot x(t) = u_0(t) + u_n(t) x^n(t),\quad x(0) &= 0
\end{align}
with $n\geq 1$ fixed. 

\begin{proposition}
 Let $u_0, u_n :[0,T]\to\RR$ be measurable, bounded functions. Then the solution of \eqref{eq:n} is 
 $x(t) = \sum_{k=0}^{\infty} x_{k}(t)$, where $x_{k}(t)$ are recursively given by $x_0(t) = \int_0^t u_0(s)\, ds$, and 
 $$x_{k}(t) = \sum_{\MultiIndex l \in N(k)} \NewtonSymbolGeneral n l k \int_0^t
  (x_0(s))^{ l_1} \cdots (x_{k-1}(s))^{ l_{k}} \, u_n(s)\, ds,$$ 
  where
 $N(k) = \SetSuchThat{(l_1,\ldots,l_{k}) \in \NN ^k}{n  =  l_1+\cdots +l_k,\, 
 k -1  =  l_2 + 2l_3 + \cdots + (k-1)l_k}.$
\end{proposition}

 Let us write the first few components of the expansion given in the above proposition.
 \begin{align*}
  x_0(t) & = \IteratedInt t (a_0) = \int_0^t u_0(s)\, ds \, , \\
  x_1(t) & = \int_0^t (x_0(s))^n \, u_n(s)\, ds \, , \\
  x_2(t) & = n \int_0^t (x_0(s))^{n-1} x_1(s) \, u_n(s)\, ds\, , \\
  x_3(t) & = \NewtonSymbol n 2 \int_0^t (x_0(s))^{n-2} (x_1(s))^2\, u_n(s)\, ds + n\int_0^t (x_0(s))^{n-1} x_2(s) \, u_n(s)\, ds\, , \\
  x_4(t) & = \NewtonSymbol n 3 \int_0^t (x_0(s))^{n-3} (x_1(s))^3\, u_n(s)\, ds 
  \\ & \quad
  + n(n-1) \int_0^t (x_0(s))^{n-2} x_2(s) x_1(s) \, u_n(s)\, ds 
  + n \int_0^t (x_0(s))^{n-1} x_3(s)  \, u_n(s)\, ds\, .
 \end{align*}

\begin{proof}
 The algebraic equation \eqref{eq:MainSeries} associated with the differential equation \eqref{eq:n} is 
 \begin{align}
 \label{eq:MainSeriesN}
 \MainSeries = a_0 + \MainSeries^{\ShuffleProduct n}\cdot a_n.  
 \end{align}
 Let us first show that the only non-vanishing homogeneous parts of $\MainSeries$ are $\MainSeries_{kn+1}$, where $k\in\NN$. We prove it by the induction on $k$. The $k$-th hypothesis is that $\MainSeries_{kn+l} =0$ for all $l = 2,\ldots,n$. For $k=0$ the hypothesis is clearly correct. Assume it is correct for $k < K$ and let us prove that $\MainSeries_{Kn+2} = \cdots = \MainSeries_{Kn+n} = 0$. Using \eqref{eq:MainSeriesN} and the induction hypothesis we see that
 \begin{align*}
  \MainSeries_{Kn+l} = \sum_{(\MultiIndex{p},\MultiIndex{m})\in \tilde N}
  C(\MultiIndex{p},\MultiIndex{m})\cdot
  \MainSeries_1^{\ShuffleProduct p_1}\ShuffleProduct\cdots \ShuffleProduct \MainSeries_{(K-1)n + 1}^{\ShuffleProduct p_{K}}\ShuffleProduct
  \MainSeries_{Kn+2}^{\ShuffleProduct m_2}\ShuffleProduct\cdots \ShuffleProduct \MainSeries_{Kn + l-1}^{\ShuffleProduct m_{l-1}}
  \cdot a_n,
 \end{align*}
 where 
 the sum is taken over all $(\MultiIndex{p},\MultiIndex{m}) = (p_1,\ldots,p_K,m_2,\ldots, m_{l-1}) \in \tilde N\subset \NN^{K+l-2}$, 
 \begin{align*}
  \tilde N  = \{n  =  p_1+\cdots +p_K &+ m_2 +\cdots+ m_{l-1}, \\ 
  Kn + l -1  & =  p_1 + \cdots + p_{K}((K-1)n+1) \\
  & \quad + m_2(Kn+2)+\cdots+ m_{l-1}(Kn + l-1)\},
 \end{align*}
 and $C(\MultiIndex{p},\MultiIndex{m})= \frac{n!}{\MultiIndex p !\, \MultiIndex m !}$, $\MultiIndex p ! = p_1!\cdots p_K!$,  $\MultiIndex m ! = m_2!\cdots m_{l-1}!$.
If $m_2 + \cdots + m_{l-1} \geq 1$, then from the second equality defining $\tilde N$ we have
\begin{align*}
 l - 1 & =  p_1 + p_2(n+1)+ \cdots + p_{K}((K-1)n+1) \\
   &\quad  + 2m_2 + \cdots + (l-1)m_{l-1} + (m_2 + \cdots + m_{l-1} -1)Kn,
\end{align*}
which is not less than $n$ (by the first equality defining $\tilde N$), a contradiction. If $m_2 + \cdots + m_{l-1} = 0$, then $p_1 + \cdots + p_K = n$ and therefore
\begin{align*}
 l - 1 = n(1 + p_2 + 2p_3 + \cdots + (K-1)p_K - K).
\end{align*}
But $n$ does not divide $l-1 \in\Set{ 1,\ldots, n-1}$. This implies $\tilde N = \emptyset$ and $\MainSeries_{kn +l} =0$ for $l = 2,\ldots, n$.
We conclude that the solution of \eqref{eq:MainSeriesN} is of the form 
$\MainSeries = \sum_{k=0}^\infty \MainSeries_{kn +1} .$

Now, similarly as above we see from \eqref{eq:MainSeriesN} that
\begin{align*}
  \MainSeries_{Kn+1} = \sum_{\MultiIndex{l}\in  N(k)}
  \frac{n!}{l_1!\cdots l_k!}\cdot\MainSeries_1^{\ShuffleProduct l_1}\ShuffleProduct
  \MainSeries_{n+1}^{\ShuffleProduct l_2}\ShuffleProduct\cdots \ShuffleProduct \MainSeries_{(k-1)n + 1}^{\ShuffleProduct l_{k}}
  \cdot a_n,
 \end{align*}
 where 
 the sum is taken over 
 \begin{align*}
   N(k)  = \SetSuchThat{(l_1,\ldots,l_{k}) \in \NN ^k}{n  =  l_1+\cdots +l_k,\, kn  =  l_1 + l_2(n+1)+ \cdots + l_{k}((k-1)n+1)}.
 \end{align*}
Using the first equation defining $N(k)$ we simplify the second equation defining $N(k)$ as follows:
\begin{align*}
 kn & = l_1 + \cdots + l_k + n(l_2 + 2l_3 + \cdots + (k-1)l_k) \\
 & = n(1 + l_2 + 2l_3 + \cdots + (k-1)l_k).
\end{align*}
Therefore,
$$N(k) = \SetSuchThat{(l_1,\ldots,l_{k}) \in \NN ^k}{n  =  l_1+\cdots +l_k,\, 
 k -1  =  l_2 + 2l_3 + \cdots + (k-1)l_k}.$$
 
Denoting $x_k(t) = \IteratedInt t (\MainSeries_{kn +1})$ and using the homomorphic property of $\IteratedInt t $ we obtain the hypothesis of the proposition.
 
 \end{proof}

\section{Comparison with the Chen-Fliess approach}
\label{sec:Comparison}
 
 In this section we compare the number of non-zero iterated integrals in two approaches: the one given in this article, and the Chen-Fliess one. Recall that in the letter approach \cite{Fliess81Fonctionnelles} we assume we have a differential equation 
 \begin{align*}
\dot x(t) &= u_0(t)\, X_0(x(t)) +  u_1(t)\, X_1(x(t)) + \cdots +  u_n(t)\, X_n(x(t)), \\
\nonumber
x(0) &= 0,
\end{align*}
where $X_i(x) = x^i \frac \partial {\partial x}$, for $i=0,\ldots,n$, are differential vector fields on $\RR$. The solution is given by
\begin{align}
\label{eq:ChenFliess}
x(t) = \sum_{v\in\WordsOfLength{}}\IteratedInt t (v) \, X_v(x)(0),
\end{align}
where for $v = a_{i_1}\cdots a_{i_k} \in \WordsOfLength {} $ we define $X_v(x)(0) := X_{i_1}\cdots X_{i_k}(x)(0)$ as a composition of vector fields acting on the function $h(x) = x$ and evaluated at the initial value $x_0 = 0$. Since $X_i(x)(0) \neq 0$ only for $i=0$ the sum can be significantly reduced. Our aim is to eliminate all unnecessary summands. Since the second derivative $\frac {\partial^2} {\partial x^2} x = 0$ we need to compute $X_v(x)(0)$ modulo the second and higher derivatives. 

\begin{lemma}
 For $k \geq 2$ and $v = a_{i_1}\cdots a_{i_k} \in\WordsOfLength {k}$ we have 
 $$X_v = i_k(i_k + i_{k-1} - 1)(i_k + i_{k-1} + i_{k-2} - 2)\cdots (i_k + \cdots + i_{2} - k+2)x^{i_1 + \cdots + i_{k} - k+1} \frac {\partial} {\partial x} \mod \frac {\partial^2} {\partial x^2}.$$
\end{lemma}

\begin{proof}
 We use the induction on $k$. The case $k =2$ is clear. Assume $w = a_{i_2}\cdots a_{i_{k+1}} \in \WordsOfLength {k}$ and $v = a_{i_1}w \in \WordsOfLength {k+1}$. By the induction hypothesis 
 $X_w = I x^{\alpha} \frac {\partial} {\partial x} \mod \frac {\partial^2} {\partial x^2},$ 
 where $I = i_{k+1}(i_{k+1} + i_{k} - 1)\cdots (i_{k+1} + \cdots + i_{3} - k+2)$, 
 $\alpha = i_2 + \cdots + i_{k+1} - k+1$, and thus 
 $X_v = x^{i_1 + \alpha-1} I \alpha \frac {\partial} {\partial x} \mod \frac {\partial^2} {\partial x^2}$. This ends the proof.
\end{proof}

Using this lemma we conclude that for $v = a_{i_1}\cdots a_{i_k} \in\WordsOfLength {k}$, $X_v(x)(0) \neq 0$ only if $i_1 + \cdots + i_{k} = k-1$, and therefore \eqref{eq:ChenFliess} simplifies to
\begin{align*}
 x(t) = \sum_{k=1}^\infty \sum_{\MultiIndex i \in M_0(k)} \IteratedInt t (a_{\MultiIndex i}) \, i_k(i_k + i_{k-1} - 1)\cdots (i_k + \cdots + i_{2} - k+2),
\end{align*}
where $a_\MultiIndex {i} := a_{i_1}\cdots a_{i_k}$, and the second sum is taken over all multi-indices $\MultiIndex i = (i_1,\ldots, i_k)$ in the set $M_0(k)\subset (\Set{0,\ldots,n})^k$ given by a one equality and  $k-1$ inequalities:
\begin{align*}
 i_k + \cdots + i_{1} - k+1 & = 0, &
 i_k + \cdots + i_{2} - k+2 & \geq 0, \\
& & & \vdots \\
& & i_k + i_{k-1} - 1 & \geq 0, \\
& & i_k  & \geq 0.
\end{align*}

On the $k$-th step of approximation there are $\Cardinality M_0(k)$ non-trivial integrals to compute. If we assume $n=\infty$ one can compute that these are Catalan numbers, i.e., $\Cardinality M_0(k) = \frac 1 {k+1} \NewtonSymbol{2k}{k}$. The first ten of these numbers are 1, 2, 5, 14, 42, 132, 429, 1430, 4862. We see that this growth is much faster than the growth 1, 2, 3, 5, 7, 11, 15, 22, 30, 42 of integrals needed to compute the $k$-th step of the expansion in our approach, as we mentioned in Remark \ref{rem:Growth}.

\section{Concluding remarks}

In the article we formulated a scheme for expanding a solution of a general non-autonomous polynomial differential equation. The time dependent homogeneous parts of the expansion were expressed in terms of iterated integrals. The formula for each of this part was given recursively by (\ref{eq:Recurence}). 
The advantage of our approach is that it is made on algebraic level. We use the shuffle product which is an algebraic analogue of multiplication of iterated integrals. 
Therefore, the algebraic formula can be easily transformed into the analytic one giving the expansion of the solution of the initial problem, as we stated in Theorem \ref{thm:Analytic}.

Finally, there is some work to be done. One way of a development is to write an explicit formula for the algebraic series $\MainSeries$ preferably with the use of shuffle product. It would also be important to find a deeper algebraic structure of this solution. Another way is to rewrite the scheme for systems of non-autonomous polynomial differential equations and estimate the radius of convergence in this case. This is important for example to integrate higher order Lie-Sheffers systems.

\section*{Acknowledgements}

The author was partially supported by the Polish Ministry of Research and Higher Education grant NN201 607540, 2011-2014.

\bibliographystyle{amsalpha}%
\bibliography{bibliography}

\begin{thebibliography}{10}

\bibitem{Alwash2007Periodic}
{\sc Alwash, M.}
\newblock {Periodic solutions of Abel differential equations}.
\newblock {\em Journal of mathematical analysis and applications 329}, 2
  (2007), 1161--1169.

\bibitem{Alwash2005Periodic}
{\sc Alwash, M.~A.}
\newblock Periodic solutions of polynomial non-autonomous differential
  equations.
\newblock {\em Electronic Journal of Differential Equations 2005}, 84 (2005),
  1--8.

\bibitem{Carinena11Riccati}
{\sc {Cari{\~n}ena}, J.~F., and {de Lucas}, J.}
\newblock {Integrability of Lie systems through Riccati equations}.
\newblock {\em Journal of Nonlinear Mathematical Physics 18}, 1 (2011), 29--54.

\bibitem{Carinena2011Lie}
{\sc Cari{\~n}ena, J., and de~Lucas, J.}
\newblock {\em Lie systems: theory, generalisations, and applications},
  vol.~479.
\newblock Institute of Mathematics, Polish Academy of Sciences, 2011.

\bibitem{Carinena99Integrability}
{\sc Carinena, J., and Ramos, A.}
\newblock {Integrability of the Riccati equation from a group-theoretic
  viewpoint}.
\newblock {\em International Journal of Modern Physics A 14}, 12 (1999),
  1935--1951.

\bibitem{Carinena11Geometric}
{\sc Cari{\~n}ena, J.~F., de~Lucas, J., and Ra{\~n}ada, M.~F.}
\newblock {A geometric approach to integrability of Abel differential
  equations}.
\newblock {\em International Journal of Theoretical Physics 50}, 7 (2011),
  2114--2124.

\bibitem{Carinena07Riccati}
{\sc Cari{\~n}ena, J.~F., Lucas, J.~d., and Ramos, A.}
\newblock A geometric approach to integrability conditions for {R}iccati
  equations.
\newblock {\em Electron. J. Differential Equations\/} (2007), No. 122, 14 pp.
  (electronic).

\bibitem{Chen68Algebraic}
{\sc Chen, K.-T.}
\newblock Algebraic paths.
\newblock {\em J. Algebra 10\/} (1968), 8--36.

\bibitem{Fliess81Fonctionnelles}
{\sc Fliess, M.}
\newblock Fonctionnelles causales non lin\'eaires et ind\'etermin\'ees non
  commutatives.
\newblock {\em Bull. Soc. Math. France 109}, 1 (1981), 3--40.

\bibitem{Gray2002Fliess}
{\sc Gray, W.~S., and Wang, Y.}
\newblock {Fliess operators on $L_p$ spaces: convergence and continuity}.
\newblock {\em Systems \& Control Letters 46}, 2 (2002), 67--74.

\bibitem{Gray2008NonCausal}
{\sc Gray, W.~S., and Wang, Y.}
\newblock {Non-causal Fliess operators and their shuffle algebra}.
\newblock {\em International Journal of Control 81}, 3 (2008), 344--357.

\bibitem{Harko2003Relativistic}
{\sc Harko, T., and Mak, M.}
\newblock {Relativistic dissipative cosmological models and Abel differential
  equation}.
\newblock {\em Computers \& Mathematics with Applications 46}, 5 (2003),
  849--853.

\bibitem{Kawski97NoncommutativePower}
{\sc Kawski, M., and Sussmann, H.~J.}
\newblock Noncommutative power series and formal {L}ie-algebraic techniques in
  nonlinear control theory.
\newblock In {\em Operators, systems, and linear algebra ({K}aiserslautern,
  1997)}, European Consort. Math. Indust. Teubner, Stuttgart, 1997,
  pp.~111--128.

\bibitem{Mak2001Solutions}
{\sc Mak, M., Chan, H., and Harko, T.}
\newblock {Solutions generating technique for Abel-type nonlinear ordinary
  differential equations}.
\newblock {\em Computers \& Mathematics with Applications 41}, 10 (2001),
  1395--1401.

\bibitem{Mak2002New}
{\sc Mak, M., and Harko, T.}
\newblock {New method for generating general solution of Abel differential
  equation}.
\newblock {\em Computers \& Mathematics with applications 43}, 1 (2002),
  91--94.

\bibitem{Pietrzkowski12Explicit}
{\sc Pietrzkowski, G.}
\newblock {Explicit solutions of the $\mathfrak{a}_1$-type Lie-Scheffers system
  and a general Riccati equation}.
\newblock {\em Journal of Dynamical and Control Systems 18}, 4 (2012),
  551--571.

\bibitem{Redheffer56Solutions}
{\sc Redheffer, R.~M.}
\newblock On solutions of {R}iccati's equation as functions of the initial
  values.
\newblock {\em J. Rational Mech. Anal. 5\/} (1956), 835--848.

\bibitem{Redheffer57Riccati}
{\sc Redheffer, R.~M.}
\newblock The {R}iccati equation: {I}nitial values and inequalities.
\newblock {\em Math. Ann. 133\/} (1957), 235--250.

\bibitem{Reutenauer93FreeLie}
{\sc Reutenauer, C.}
\newblock {\em Free {L}ie algebras}, vol.~7 of {\em London Mathematical Society
  Monographs. New Series}.
\newblock The Clarendon Press Oxford University Press, New York, 1993.
\newblock Oxford Science Publications.

\bibitem{Xu2011Short}
{\sc Xu, Y., and He, Z.}
\newblock {The short memory principle for solving Abel differential equation of
  fractional order}.
\newblock {\em Computers \& Mathematics with Applications 62}, 12 (2011),
  4796--4805.

\end{thebibliography}

\end{document}